\documentclass[11pt]{amsart}
\usepackage{euscript}
\usepackage{amssymb}
\usepackage{amsmath}
\usepackage{enumitem}
\usepackage{epic}
\usepackage{graphics}
\usepackage{epsfig}
\usepackage{color}

\usepackage{parskip}
\usepackage{amscd,euscript}
\usepackage[frame,cmtip,curve,arrow,matrix,line,graph]{xy}
\usepackage{setspace}
\usepackage{tikz}
\usepackage{tikz-cd}

\tikzset{
	cross/.pic = {
		\draw[rotate = 45] (-#1,0) -- (#1,0);
		\draw[rotate = 45] (0,-#1) -- (0, #1);
	}
}

\usepackage[breaklinks,colorlinks,citecolor=teal,linkcolor=teal,urlcolor=teal,pagebackref,hyperindex]{hyperref}
\usepackage[margin=1.2in]{geometry}

\newtheoremstyle{my}{1.5em}{0.5em}{\em}{}{\sc}{.}{0.5em}{}

\newtheorem{theorem}{Theorem}[section]
\newtheorem*{theorem*}{Theorem}
\newtheorem{thm}{Theorem}[section]

\newtheorem*{Theorem*}{Theorem}

\newtheorem{corollary}[thm]{Corollary}
\newtheorem*{corollary*}{Corollary}

\newtheorem*{prop*}{Proposition}

\newtheorem*{conjecture*}{Conjecture}

\newtheorem*{question*}{Question}
\newtheorem{defn}[thm]{Definition}

\newtheorem*{definitions*}{Definitions}

\newtheorem*{rem*}{Remark}
\newtheorem{Remark}[thm]{Remark}

\newtheorem*{remark*}{Remark}

\newtheorem*{remarks*}{Remarks}
\newtheorem*{example*}{Example}

\newtheorem*{examples*}{Examples}

\newtheorem*{convention*}{Convention}
\newtheorem*{conventions*}{Conventions}

\newtheorem*{exercise*}{Exercise}
\newtheorem*{bibliographical-note*}{Bibliographical note}

\newtheorem{cor}[thm]{Corollary}
\newtheorem{lemma}[thm]{Lemma}

\newtheorem{remark}[thm]{Remark}

\newcommand{\scrH}{\EuScript{H}}

\newcommand{\scrA}{\EuScript{A}}

\newcommand{\scrX}{\EuScript{X}}

\newcommand{\scrS}{\EuScript{S}}

\newcommand{\bA}{\mathbb{A}}

\newcommand{\bR}{\mathbb{R}}
\newcommand{\bZ}{\mathbb{Z}}
\newcommand{\bQ}{\mathbb{Q}}
\newcommand{\bC}{\mathbb{C}}
\newcommand{\bN}{\mathbb{N}}
\newcommand{\bP}{\mathbb{P}}

\newcommand{\Sym}{\mathrm{Sym}}

\renewcommand{\hbar}{\overline{\frak{h}}}
\newcommand{\Aut}{\mathrm{Aut}}

\newcommand{\bK}{\mathbb{K}}

 \numberwithin{equation}{section}

\renewcommand{\leq}{\leqslant}
\renewcommand{\geq}{\geqslant}

\newcommand{\C}{\mathbb C}

\newcommand{\bT}{\mathbb{T}}

\newcommand{\ev}{\operatorname{ev}}

\newcommand{\Spec}{\operatorname{Spec}}

\newcommand{\Gr}{\mathrm{Gr}}

\renewcommand{\sc}{\operatorname{sc}}

\newcommand{\Mod}{\operatorname{Mod}}

\title[A Stably Degenerate Singularity]{Isolated hypersurface singularities may be \\ stably degenerate}
\author{Mark McLean}
\address{Mark McLean, Department of Mathematics, Stony Brook University, Stony Brook, NY 11794-3651, USA.}
\email{mark.mclean@stonybrook.edu}
\author{Ivan Smith}
\address{Ivan Smith, Centre for Mathematical Sciences, University of Cambridge, Wilberforce Road, Cambridge CB3 0WB, U.K.}
\email{is200@cam.ac.uk}

\begin{document}

\begin{abstract}
	We prove the existence of stably degenerate hypersurface singularities. Ingredients in the proof include the mixed Hodge structure on the contact locus of a singularity and that  moduli spaces of polarised abelian varieties of large dimension are of general type.
\end{abstract}

\maketitle
\begingroup
\setlength{\parskip}{0pt} 
{ \setcounter{tocdepth}{4}
	\setcounter{secnumdepth}{4}
	\hypersetup{linkcolor=black}
	\tableofcontents
}
\endgroup

\section{Introduction}

A \emph{stably degenerate hypersurface singularity} is a hypersurface singularity whose associated defining polynomial is not stably equivalent to a Newton non-degenerate polynomial over $\bC$ (See Definitions \ref{defnnewtonnondegenerate}, \ref{defnequivalence}, \ref{defnstabilisation} and \ref{defnstablyequivalent}).

In this paper we will prove the following theorem.

\begin{theorem} \label{maintheorem}
	There exists a stably degenerate isolated hypersurface singularity.
\end{theorem}

This answers a problem, originally stated by Arnol'd in 1975, asking if there are stably degenerate isolated hypersurface singularities (see \cite[Problem 3]{Arnold1983} and \cite[Problems 1975-3, 1976-8]{Arnold2005}). (It is well-known that there are isolated hypersurface singularities\footnote{For an example, consider $f(x,y)=x^2y^2(x+y)^2+ x^7+y^7$.
	The singular locus of the leading homogeneous part in any coordinate system always contains three distinct lines, so must  intersect the open orbit $(\bC^*)^2$.} which are degenerate in any local analytic co-ordinates, if one does not allow stabilisation.)

Briefly, our example has the following shape: we take $f=F+h$, where $F$ is homogenous of degree $d$ defining a hypersurface in projective space birational to a sufficiently general abelian variety $A$ of high dimension, and $h$ is a polynomial (for instance of Fermat type) all of whose monomials are of degree greater than the degree of $F$ and such that $f$ has an isolated singularity.

To show our example is stably degenerate, we show that a weight graded piece of the mixed Hodge structure of the $d$-contact locus of a stabilisation of $f$ (Definition \ref{defnmodifiedcontactlocus}) corresponds to a polarised abelian variety $A'$ isogenous to $A$  (Lemma \ref{lemmafkHodgeComputation}).
Suppose, for a contradiction, a stabilisation of $f$ was Newton non-degenerate after a change of coordinates. Then $f$ would sit inside the family $(f_\tau)_{\tau \in T}$ of Newton non-degenerate singularities
with the same Newton polyhedron and with a large degree bound.
Taking an appropriate graded piece of the mixed Hodge structure of the $d$-contact locus would give us a map from $T$ to a space of  polarised abelian varieties of high dimension (See Lemma \ref{lemmaTopTrivual}). Such a map must be constant since $T$ is rational and rational curves cannot pass through the (chosen to be sufficiently general) point corresponding to $A'$ in that moduli space.
But this contradicts the fact that there can only be countably many abelian varieties that are the image of such a constant map (cf.  Lemma \ref{lemmaconditionsforA}).

Our argument is essentially an existence theorem, though we briefly discuss how one might (impractically!) algorithmically find a concrete example of a stably degenerate singularity in Section \ref{sec:algorithm}. Theorem \ref{maintheorem}, taken together with \cite[Corollary 3.2]{stevens2021conjectures}, implies that there must exist some cubic polynomial (perhaps in very many variables) which is stably degenerate.

Stevens' article \cite{stevens2021conjectures} gives a survey on Arnol'd's question. Stevens conjectured that, in characteristic zero, a singularity of Milnor number $\mu$ whose $\mu=\mathrm{constant}$ stratum was globally singular would have generic member  being stably degenerate. Our argument also uses a global geometric feature of a  moduli space related to the singularity, but for us the input is non-uniruledness rather than non-smoothness.

We work throughout over the field $\bC$.

\noindent \emph{Acknowledgements.} The authors are grateful to Morihiko Saito, Tony Scholl and Burt Totaro for helpful discussions.  M.M. is partially supported by NSF grant DMS-2203308. I.S. is partially supported by UKRI Frontier Research grant EP/X030660/1 (in lieu of an ERC advanced grant).

\section{Basic Definitions}

\begin{defn} \label{defnnewtonnondegenerate} \cite[Sections 6.2.1 and 6.2.2]{arnold2012singularities}.
	Let $z_1,\cdots,z_N$ be the standard coordinates on $\bC^N$.
	For each tuple $\nu = (\nu_1,\cdots,\nu_N) \in \bN^N$, define $z^\nu := \prod_{j=1}^N z_j^{\nu_j}$.
	Let
	\begin{equation}
		g : \C^N \to \C, \quad g = \sum_{\nu \in S} a_\nu z^\nu
	\end{equation}
	be a polynomial where $S \subset \bN^N$ is a finite set and $(a_\nu)_{\nu \in S}$ are elements of $\bC^*$.
	For any subset $S' \subset \bR^N$, we define $g_{S'} := \sum_{\nu \in S \cap S'} a_\nu z^\nu$.
	We define the \emph{Newton polyhedron $\Gamma_+(g)$ of $g$} to be the convex hull of the subsets $\nu + \bR_{\geq 0}^N \subset \bR^N$, $\nu \in S$.
	A \emph{face} of $\Gamma_+(g)$ is a subset of the form \begin{equation}
		\Delta_L := \left\{x \in \Gamma_+(f) \ : \ L(x) = \inf_{y \in \Gamma_+(f)} L(y) \right\}
	\end{equation}
	where $L : \bR^N \to \bR$ is linear.
	We say that $g$ is \emph{Newton non-degenerate}
	if $dg_\Delta$ does not vanish at any point in $(\bC^*)^N$ for each compact face $\Delta \subset \Gamma_+(g)$.
\end{defn}

\begin{remark}
	The notion of being Newton non-degenerate depends strongly on the choice of co-ordinates.
\end{remark}

\begin{remark}
	The polynomial $g$ is Newton non-degenerate iff, in the toric variety associated with a regular unimodular subdivision of the normal fan of its Newton polyhedron, the strict transform of the hypersurface $\{g=0\}$ meets every torus orbit over $0 \in \bC^N$ transversely.
\end{remark}

\begin{defn} \label{defnequivalence}
	Let $f, g : \C^N \to \C$ be two polynomials.
	Let $\widehat{f}, \widehat{g}$ be the corresponding elements of the completion $R := \C[[z_1,\cdots,z_N]]$.
	We say that $f$ is \emph{equivalent} to $g$ if there is an automorphism $\Phi$ of $R$ fixing the ideal $(z_1,\cdots,z_N)$ sending $\widehat{f}$ to $\widehat{g}$.
\end{defn}

\begin{defn} \label{defnstabilisation}
	Let $g : \bC^N \to \bC$ be a polynomial.
	For each $k \in \bN$, define a \emph{$k$-stabilisation} of $g$ to be a polynomial:
	\begin{equation}
		g_Q : \bC^k \times \bC^N \to \bC, \quad g_Q(w_1,\cdots, w_k,z_1,\cdots,z_N) = Q(w_1,\cdots,w_k) + g(z_1,\cdots,z_N)
	\end{equation}
	where $Q$ is a non-degenerate quadratic form in $k$ variables.
	A \emph{stabilisation} of $g$ is a $k$-stabilisation of $g$ for some $k \in \bN$.

	For each $k \in \bN$, define:
	\begin{equation}
		g_k : \bC^k \times \bC^k \times \bC^N \to \bC, \quad g_k(u,v,z) := uv + g(z)
	\end{equation}
	where $uv$ means the dot product.
\end{defn}

Since all non-degenerate quadratic forms are equivalent, we have the following lemma.
\begin{lemma} \label{equivalentstandard}
	Any two $k$-stabilisations of a polynomial are equivalent for each $k \in \bN$.
	In particular, the $2k$-fold stabilisation of any polynomial $g$ is equivalent to $g_k$ for each $k \in \bN$.
\end{lemma}

\begin{lemma} \label{lemmastabilisationofNND}
	If $g$ is Newton non-degenerate then the polynomial
	\begin{equation}
		\widetilde{g} : \bC \times \bC^N, \ (w,z) \to w^2 + g(z)
	\end{equation}
	is Newton non-degenerate.
\end{lemma}
\begin{proof}
	Let us embed $\Gamma_+(g)$ into $\bR^N = 0 \times \bR^N \subset \bR^{N+1}$ which is the hyperplane given by setting the first coordinate to $0$.
	Then each compact face of $\Gamma_+(\widetilde{g})$ is either a compact face $\Delta$ of $\Gamma_+(g)$ or the convex hull $\widetilde{\Delta}$ of a compact face $\Delta$ of $\Gamma_+(g)$ and the point $(1,0,\cdots,0)$.
	We have $d\widetilde{g}_\Delta = dg_\Delta$ and $d\widetilde{g}_{\widetilde{\Delta}} = 2wdw + dg_\Delta$. In both cases they do not vanish at any point of $(\C^*)^{N+1}$.
\end{proof}

\begin{defn} \label{defnstablyequivalent}
	Two polynomials $f$ and $g$ are \emph{stably equivalent} if a stabilisation of $f$ is equivalent to a stabilisation of $g$.
\end{defn}

Arnol'd asked whether every polynomial with an isolated singularity at $0$ was stably equivalent to one which is Newton non-degenerate. We prove the existence of  counterexamples to that statement in the subsequent sections. By Lemma \ref{lemmastabilisationofNND}, it is sufficient to construct a polynomial $f$ for which none of the particular stabilisations $f_k$ (by $2k$ variables) are analyically equivalent to Newton non-degenerate singularities.

\section{Construction of the Singularity} \label{sectionConstruction}

We will fix the following notation throughout the rest of the paper. Let $\scrA_g$ denote the moduli space of principally polarised complex abelian varieties of dimension $g$. We typically view $\scrA_g$ as a complex orbifold\footnote{In the final section it will be important to view $\scrA_g$ as a scheme defined over $\bQ$ and we may write $\scrA_g(\bC)$ to emphasise when we are taking complex points.}. If $\delta = (d_1,\ldots,d_g) \in \bN^g$, we write $\scrA_{g,\delta}$ for the moduli space of abelian varieties of dimension $g$ with polarisation of type $\delta$; so $\delta = \underline{1} = (1,\ldots,1)$ corresponds to principal polarisations.  Note that the dual $A^{\vee}$ of a principally polarised abelian variety is again principally polarised, but an abelian variety isogenous to a principally polarised one need not admit a principal polarisation.

\begin{remark}\label{rmk:abvar-as-Hodge}
	We will make  use of the standard equivalence of categories \cite{deligne1971theorie2}  between complex abelian varieties up to isogeny and polarisable $\bQ$-Hodge structures of type $(-1,0) + (0,-1)$, via the map $A \mapsto H_1(A;\bQ) = H^1(A;\bQ)^{\vee}$; and of the analogous equivalence between (principally) polarised complex abelian varieties and (unimodular) polarised $\bZ$-Hodge structures of type $(-1,0) + (0,-1)$. As usual if $V$ is a  Hodge structure then $V(m)$ denotes its $m$-th Tate twist.
\end{remark}

\begin{defn} \label{defnweightedhomogenous}
	For any subfield $\bK \subset \bC$, we define $\scrS(\bK)$ to be the set of isomorphism classes of morphisms of varieties over $\bK$ (i.e. diagrams $p : X \to T$ where $p,V,T$ are all defined over $\bK$).
	We have a natural map
	\begin{equation} \label{eqndefiningoverbC}
		\scrS(\bK) \to \scrS(\bC), \quad [p : X \to T] \to [p_\bC : X_\bC \to T_\bC]
	\end{equation}
	sending each morphism to the corresponding one over $\bC$.

	For any smooth projective variety $V$ and element $[p : X \to T]$ of $\scrS(\bK)$, we say that $p$ is \emph{generically $H^1(V)$-constant} if there exists integers $e,i,m$ and a Zariski dense subset $U \subset T_\bC$
	so that the for every $\tau \in U$, there exists an isomorphism of Hodge structures
	\begin{equation}
		\Gr_e^W(H_c^i(p_\bC^{-1}(\tau);\bQ)(m)  \cong H^1(V;\bQ)
	\end{equation}
	(i.e. a graded piece of the mixed Hodge structure of a generic fiber agrees with $H^1(V;\bQ)$ up to a shift).
\end{defn}

\begin{lemma} \label{lemmaconditionsforA}
	For each $g \geq 16$, there is an abelian variety $A$ in $\scrA_g$ with the following property. Let $A'$ be an  abelian variety isogenous to $A^{\vee}$, with a polarisation of type $\delta$. Then
	\begin{enumerate}
		\item \label{itemNotInRational}
		      $A'$ is not contained in any rational subvariety of $\scrA_{g,\delta}$;
		\item \label{itemNoConstantSummandOfToric}
		      no element of $\scrS(\bQ)$ is generically $H^1(A')$-constant and
		\item \label{itemNScondition} the N\'eron-Severi  group of $A'$ has rank $1$.
	\end{enumerate}
\end{lemma}

\begin{proof}
	According to \cite{mumford2006kodaira}, the moduli space of principally polarised abelian varieties has general type once $g\geq 7$, so the general point of $\scrA_g$ does not lie on any rational subvariety. The corresponding result for the spaces $\scrA_{g,\delta}$ with not necessarily principal polarisation holds for every $\delta$ once $g \geq 16$, see \cite{Tai}.

	Similarly, for any polarisation $\delta$ and any $g$, the general abelian variety has Neron-Severi group of rank 1. (This holds whenever $\mathrm{End}_{\bQ}(A) = \bQ$, and holds away from countably many subvarieties, cf. \cite[Chapters 5,9]{BL}. Note that having Picard rank 1 is preserved by isogenies, and hence by duality.)

	An isogeny $g: A' \to A$ is polarised if the polarisations $\lambda_A$ and $\lambda_{A'}$ satisfy $g^*\lambda_A = n\cdot \lambda_{A'}$ for some $n\in\bN$. The locus $\scrH_{g,\delta} \subset \scrA_g \times \scrA_{g,\delta}$ of pairs $(A,A')$ related by a polarised isogeny is a union of algebraic Hecke correspondences \cite[Lemma 4.2]{Orr} (to specify a particular one requires fixing $n$ and keeping track of the map underlying the isogeny). Each such irreducible correspondence is finite over both factors. Let $W_{g,\delta} \subset \scrA_{g,\delta}$ denote the (countable, proper) union of all positive-dimensional rational subvarieties. Let $W_g \subset \scrA_g$ denote the countable union of varieties which are Hecke correspondence images of points of $W_{g,\delta}$, for some $\delta$. Taking the diagonal correspondence, this in particular means $W_g$ contains the rational subvarieties of $\scrA_g$ itself.

	Since $\scrS(\bQ)$ is countable, there are only countably many abelian varieties $A''$ with the property that there exists some element of $\scrS(\bQ)$ that is generically $H^1(A'')$-constant.
	For any chosen polarisation $\delta$, let $Z_{g,\delta} \subset \scrA_{g,\delta}$ denote this countable subset. Let $Z_g \subset \scrA_g$ be the countable subset given by Hecke correspondence images of the various $Z_{g,\delta}$.

	If $A$ has Picard rank one and $A'$ is isogenous to $A$, then $A'$ has Picard rank one and the isogeny is necessarily polarised. We now pick any abelian variety $A$ of Picard rank one  which does not belong to $Z_g \cup W_g$.  This excludes a countable family of proper subvarieties. Any such $A$ then satisfies the required property, i.e. all three constraints hold on all members of its isogeny class.
\end{proof}

\begin{lemma}
	Every complex variety is birational to a hypersurface in projective space.
\end{lemma}

\begin{proof} This holds over any algebraically closed field; see \cite[I.4.9]{Hartshorne}. \end{proof}

We let $H \subset \bP^{N-1}$ be a hypersurface birational to $A$ where $N-1=g$ and let $F : \bC^N \to \C$ be its defining homogeneous polynomial of degree $d$.

\begin{lemma}\label{lem:add to make isolated}
	If $F: \bC^N \to \bC$ is homogeneous of degree $d$, then one can add a polynomial $h$ of degree $>d$ to obtain a polynomial $f = F + h$ which has an isolated singularity at $0$.
\end{lemma}

\begin{proof}
	One can check that for any $D > d$, taking $h = \sum_{j=1}^N z_j^{D}$ suffices.  (The common zero set of the partial derivatives of $f=F+h$ has no critical points at infinity, so is affine of dimension zero and hence finite, so $0$ is an isolated point; alternatively, the Jacobi algebra of $f$ is clearly finite-dimensional.)
\end{proof}

Now pick $f : \bC^N \to \bC$  as in Lemma \ref{lem:add to make isolated}, so that $f$ has an isolated singularity at $0$, i.e. $f = F + h$ where $F$ defines a hypersurface birational to an $A$ as in Lemma \ref{lemmaconditionsforA}. This $f$ is the singularity which we will prove is stably degenerate.

\section{Contact Loci} \label{defnofcontactloci}

We recall the definition of the \emph{contact locus} of a singularity, cf. \cite{ELM}. We fix throughout an integer $d \in \bN$ which we then suppress from the notation.

\begin{defn}  \label{defnmodifiedcontactlocus}
	For any (algebraic or analytic) variety $V$, define $\scrX(V)$ to be the space of $d$-jets of morphisms $\bC \to V$.
	We define the \emph{evaluation map} $\ev : \scrX(V) \to V$ to be the map sending a jet $u(t)$ to $u(0)$.

	For any polynomial $g : \C^m \to \C$, we define the \emph{$d$-contact locus of $g$} to be the subspace $\chi_g \subset \scrX(\bC^m)$ of jets $u(t)$ satisfying $u(0) = 0$ and $u(t) = a t^d \ \Mod \ t^{d+1}$ for some $a \in \bC^*$.
	We define
	$\chi'_g \subset \scrX(g^{-1}(0))$ to be the subspace of jets
	$u(t)$ satisfying $\ev(u(t))=0$.
	We also define $\chi''_g := \chi_g \cup \chi'_g \subset \scrX(\bC^m)$.
\end{defn}

The contact locus behaves well under formal changes of co-ordinates.

\begin{lemma} \label{lemmaChangeOfCoordinatrd}
	Let $h_0,h_1$ be two polynomials which are equivalent in the sense of \ref{defnequivalence}.
	Then $\chi'_{h_0}$ is isomorphic to $\chi'_{h_1}$.
\end{lemma}
\begin{proof}
	This isomorphism is given by postcomposing each jet in $\chi'_{h_0}$ by the corresponding change of coordinates.
\end{proof}

\begin{defn}
	Let $(f_\tau)_{\tau \in T}$ be a family of polynomials $\bC^N \to \bC$ parameterized by a variety $T$. We define $\chi_{f_T} \subset \scrX(\bC^N) \times T$ to be the subvariety $\{ (u,\tau) \in \scrX(\bC^N) \times T \, | \, u \in \chi_{f_\tau} \}$.
	We define
	\begin{equation} \label{eqnprojectionmap}
		\pi_{f_T} : \chi_{f_T} \to T, \quad \pi_{f_T}(u,\tau) = \tau.
	\end{equation}
\end{defn}

\section{Cohomological Triviality of a Jet Space Fibration}

Let $p:\bC^N \to \bC$ be a polynomial with Newton polyhedron $\Gamma_+(p)$.

\begin{lemma}
	A regular unimodular subdivision of the dual fan $\Sigma_{\Gamma}$ of $\Gamma_+(p)$ defines a toric resolution $Y \to \bC^N$. If $p$ is Newton non-degenerate, this is a log resolution $(Y,D) \to (\bC^N,p^{-1}(0))$.
\end{lemma}

\begin{proof}
	Standard, see e.g. \cite{CLS}.
\end{proof}

\begin{cor}\label{cor:simultaneous}
	An algebraic family $(f_\tau)_{\tau \in T}$ of Newton non-degenerate polynomials $\C^N \to \C$ with the same Newton polyhedron admits a simultaneous log resolution $Y \times T \to \bC^N \times T$.
\end{cor}

\begin{defn} \label{defntorusandcontact}
	We define $\bT := (\bC^*)^N$. We define $\scrX(\bT)$ to be the group where multiplication is induced by multiplication on $\bT$:
	\begin{equation}
		\scrX(\bT) \times \scrX(\bT) = \scrX(\bT \times \bT) \to \scrX(\bT).
	\end{equation}
	For each $\nu=(\nu_1,\cdots,\nu_N) \in \bN^N$ let $\scrX_\nu \subset \scrX(\bC^N)$ be the subspace of jets that vanish to order exactly $\nu_i$ along the $i$th coordinate hyperplane for each $i=1,\cdots,N$.
\end{defn}

\begin{lemma} \label{lemmaorbitsofscrXbT}
	Let $\bT$ act on $\bC^N$ in the usual way via diagonal matrices.
	Then for each orbit $O$ of the induced action of $\scrX(\bT)$ on $\scrX(\bC^N)$, there exists $\nu=(\nu_1,\cdots,\nu_N) \in \bN^N$ so that $O = \scrX_\nu$.
	Additionally, $\ev(O) = \{0\}$ if and only if $\nu \in \bN^N_{>0}$ under this correspondence.
\end{lemma}
\begin{proof}
	It is sufficient to prove this in the case $N=1$ since the group action and $\scrX_\nu$ splits up as a product.
	In this case, the group action is the map:
	\begin{equation}
		\scrX(\bT) \times \scrX(\bC) \to \scrX(\bC), \quad (v(t),u(t)) \to v(t)u(t).
	\end{equation}
	The orbits of this action are $\scrX_\nu$, $\nu \in \bN$ since every element of $\scrX_\nu$ is of the form $t^\nu w(t)$ where $w(t) \in \scrX(\bC^*)$.
	Also such orbits map to $0 \in \bC$ under $\ev$ if and only if $\nu \neq 0$.
\end{proof}

Let $E = \{\prod_{i=1}^N z_i = 0\} \subset \bC^N$ where $z_1,\cdots,z_N$ are the standard coordinates of $\bC^N$. An analytic submanifold of $\bC^N$ is said to be transverse to $E$ if it is transverse to every smooth torus orbit in $E$.

\begin{lemma} \label{lemmaarcspaceintersection}
	Suppose $(g_\tau)_{\tau \in T}$ is a family of analytic functions defined on some open subset $U$ of $\bC^N$ (in the Euclidean topology)
	so that $0$ is a regular value of $g_\tau$ and $g_\tau^{-1}(0)$ is smooth and transverse to $E$ for each $\tau \in T$.
	Let $O$ be an orbit of the $\scrX(\bT)$-action on $\scrX(\bC^N)$ so that $\ev(O) \subset E$.
	Let $k_1,\cdots,k_N \in \bN$.
	Define $h_\tau := g_\tau \prod_{j=1}^N z_i^{k_i}$ for each $\tau \in T$.
	Then the map
	\begin{equation} \label{eqnprojectiontoT}
		\{ (v(t),\tau) \in \scrX(U) \times T \ : \ h_\tau(v(t)) = a t^d \ \Mod \ t^{d+1}, \ a \in \bC^* \} \cap (O \times T) \to (\ev(O) \cap U) \times T
	\end{equation}
	sending $(v(t),\tau)$ to $(\ev(v(t)),\tau)$ is a topologically locally trivial fibration over
	\begin{equation} \label{eqnfibrationregion}
		\{(y,\tau) \in (\ev(O) \cap U) \times T. \ : \ g_\tau(y)=0\}
	\end{equation}
	or its complement in $(\ev(O) \cap U) \times T$.
\end{lemma}
\begin{proof}
	This is a local statement, so we can assume that $U$ is a small ball and $T$ can be replaced with a small open subspace in the analytic topology.
	If these neighborhoods are sufficiently small, then after a $T$ dependent analytic change of coordinates on $U$ fixing $E$, we can assume that $g_\tau = z_1 - 1$ by our transversality assumption.
	Since this change of coordinates fixes $E$, the orbits of the $\scrX(\bT)$-action on $\scrX(\bC^N)$ intersected with $U$ are permuted by Lemma \ref{lemmaorbitsofscrXbT} and so it is sufficient to prove our Lemma in this new setting.
	Also by Lemma \ref{lemmaorbitsofscrXbT}, we have  $O = \scrX_\nu$ for some $\nu = (\nu_1,\cdots,\nu_N) \in \bN^N$.
	Hence, the domain of \eqref{eqnprojectiontoT}
	is the space of pairs $(v(t),\tau) \in \scrX(U) \times T$ with
	the property $v(t) \in \scrX_\nu$ and
	\begin{equation} \label{eqnhtvanishes}
		h_\tau(v(t)) = a t^d \ \Mod \ t^{d+1}, \ a \in \ \bC^*.
	\end{equation}
	Each $v(t) \in X_\nu$ is of the form
	\begin{equation} \label{vtinnut}
		v(t) = (t^{\nu_1} v_1(t),\cdots, t^{\nu_N} v_N(t))
	\end{equation}
	where $v_1(t),\cdots,v_N(t)$ live in $\scrX(\bC^*)$.
	Therefore, the domain of \eqref{eqnprojectiontoT} is the space of pairs $(v(t),\tau) \in \scrX(U) \times T$ with $v(t)$ written as in \eqref{vtinnut} and $v_j(t) \in \scrX(\bC^*), \ j=1,\cdots,N$, satisfying
	\begin{equation}
		t^{\nu_1} v_1(t) - 1 = a t^{d- \sum_{j=1}^N k_j \nu_j} \ \Mod \ t^{d+1 - \sum_{j=1}^N k_j \nu_j}
	\end{equation}
	for some $a \in \C^*$.
	If $\nu_1 \neq 0$, then this variety is either empty, or there are no constraints on $v_1(t),\cdots,v_N(t)$
	and $\tau$ except the requirement that $\ev(v(t)) \in U$.
	If $\nu_1 = 0$, then these equations tell us that $v_1(t)-1$ vanishes at $0$ to order exactly $d- \sum_{j=1}^N k_j \nu_j$ and there are no constraints on $v_2(t),\cdots,v_N(t)$ and $\tau$ except the requirement that $\ev(v(t)) \in U$.
	Hence \eqref{eqnprojectiontoT} is a topologically locally trivial fibration over \eqref{eqnfibrationregion} or its complement in $(\ev(O) \cap U) \times T$.
\end{proof}

\begin{lemma} \label{lemmalocallyconstantinductionstep}
	Let $p: X \to T$ be a morphism and $Z \subset X$ a closed subvariety and $U = X-Z$. Suppose that $R(p|_Z)_!(\bQ_Z)$ and $R(p|_U)_!(\bQ_U)$ are (analytically) locally constant sheaves. Then so is $Rp_!\bQ_X$.
\end{lemma}
\begin{proof}
	We have an exact sequence:
	\begin{equation} \label{pushforwardexact}
		R^{i-1}(p|_Z)_!(\bQ_Z) \to
		R^i(p|_U)_!(\bQ_U) \to
		R^ip_!\bQ_X \to
		R^i(p|_Z)_!(\bQ_Z) \to
		R^{i+1}(p|_U)_!(\bQ_U).
	\end{equation}
	Since locally free sheaves form a weak Serre subcategory of all constructible sheaves (See
	\cite[Proposition 2.5]{lunts2026categories} and \cite[02MO, 0754]{stacks-project}),
	we get that the middle term term in \eqref{pushforwardexact} is locally free.
\end{proof}

\begin{corollary} \label{corollarystratifiedfibration}
	Let $p : Y \to T$ be a morphism of varieties. Suppose that $Y$ admits a stratification so that
	$R(p|_S)_!(\bQ_S)$ is a locally constant sheaf for each stratum $S$.
	Then $Rp_!\bQ_Y$ is a locally constant sheaf.
\end{corollary}
\begin{proof}
	Let $S_i \subset X$ be the union of all the strata of dimension $i$ and $S_{\leq i}$ the union of all the strata of dimension $\leq i$ for each $i \in \bN$.
	Suppose (by induction) $R(p|_{S_{\leq i}})_!(\bQ_{S_{\leq i}})$ is locally constant for some $i$.
	Then by Lemma \ref{lemmalocallyconstantinductionstep} with $X=S_{\leq i+1}$, $Z=S_{\leq i}$ and $U=S_{i+1}$, we have $R(p|_{S_{\leq i+1}})_!(\bQ_{S_{\leq i+1}})$ is locally constant. We are now done by induction since $Y=S_{\leq k}$ for some $k$.
\end{proof}

\begin{lemma} \label{lemmaTopTrivual}
	Let $(f_\tau)_{\tau\in T}$ be an algebraic family of Newton non-degenerate polynomials $\C^N \to \C$ with the same Newton polyhedron. Then $(R\pi_{f_T})_!\bQ_{\chi_{f_T}}$ is a locally constant sheaf on $T$.
\end{lemma}
\begin{proof}
	We will show $\chi_{f_\tau}$ admits a stratification so that $R(\pi_{f_T}|_S)_!(\bQ_S)$ is locally constant for each stratum $S$. Our result will then follow from Corollary \ref{corollarystratifiedfibration}.

	Let $f: \bC^N \times T \to \bC$ be the map $(z,\tau) \mapsto f_\tau(z)$ and choose a simultaneous toric log resolution as in Corollary \ref{cor:simultaneous}.
	Let $O$ be an orbit of the $\scrX(\bT)$ action on $\scrX(\bC^N)$ so that $\ev(O) = \{0\}$.
	Now $\scrX(\bT)$ acts on $\scrX(Y)$ too and there is an orbit $O'$ of this group action that maps equivariantly to $O$.
	Let $K \subset \scrX(\bT)$ be the stabilizer group of any point of $O$ and $K' \subset K$ the stabilizer group of any point of $O'$.
	Let $\chi \subset \scrX(Y) \times T$ be the subspace of pairs $(v(t),\tau)$ where $\tau \in T$ and $v(t)$ maps to $\chi_{f_\tau}$.
	Let $\chi_{O'} := \chi \cap (O' \times T)$
	and let $\chi_O := \chi_{f_T} \cap (O \times T)$.
	Let $p' : \chi_{O'} \to T$ and $p : \chi_O \to T$ be the natural projection maps.

	Since $Y \times T$ is a simultaneous log resolution of the map
	\begin{equation}
		f_T : Y \times T \to \bC, \ f_T(y,\tau) = f_\tau(y),
	\end{equation}
	we get that $p'$ is a topologically locally trivial fibration by Lemma \ref{lemmaarcspaceintersection} applied to a toric chart of $Y$ containing the $\bT$-orbit $\ev(O')$.
	Hence $Rp'_!\bQ_{\chi_{O'}}$ is locally constant.

	Also $\chi_O$ is the quotient of $\chi_{O'}$ by the fiber preserving proper free action of $K/K'$.
	Let $Q : \chi_{O'} \to \chi_O$ be this quotient map.
	Since each fiber of $Q$ is isomorphic to a product of copies of $\bC^*$ and $\bC$, we get that $RQ_!\bQ_{\chi_{O'}}$ is a direct sum of shifts of $\bQ_{\chi_O}$.
	Hence $Rp'_!\bQ_{\chi_{O'}}$
	is a direct sum of shifts of $Rp_!\bQ_{\chi_O}$.
	This implies that $Rp_!\bQ_{\chi_O}$ is the kernel of an idempotent endomorphism of $Rp'_!\bQ_{\chi_{O'}}$.
	Hence $Rp_!\bQ_{\chi_O}$ is locally constant.

	Hence by Lemma \ref{lemmaorbitsofscrXbT},
	$R(\pi_\nu)_! \bQ_{\chi_\nu}$ is locally constant where
	$\chi_\nu := \chi_{f_T} \cap (\scrX_\nu \times T)$ and
	$\pi_\nu : \chi_\nu \to T$ is the natural projection map for each $\nu \in \bN^N_{>0}$.
	Since the spaces $(\chi_\nu)_{\nu \in \bN^N_{>0}}$ stratify $\chi_{f_T}$, our result follows from Corollary \ref{corollarystratifiedfibration}.
\end{proof}

\section{Mixed Hodge Structure Computations}\label{sec:mhs}

Let us fix a polynomial $g : \bC^N \to \bC$ whose degree $<d$ terms are all $0$ and whose degree $d$ term is $G$.
Then for each $k \in \bN$, $g_k$
denotes the $2k$-fold stabilisation, as in Definition \ref{defnequivalence}.

\begin{defn}
	If $u,v \in \bC^N$, recall that  $uv$ denotes their dot product.
	Similarly if $u(t),v(t) \in \scrX(\bC^N)$, then we define $u(t)v(t) \in \scrX(\bC)$ to be their dot product.
\end{defn}

Recall the various subspaces of  contact loci  from Definition \ref{defnmodifiedcontactlocus}.

\begin{lemma} \label{LemmaChiComputation}
	For each $k \in \bN$,
	$\chi_{g_k}$ is isomorphic to
	\begin{equation} \label{eqnMainChiEqn}
		\begin{aligned}
			\bC^{(d-1)N + 2k} \times \Big\{ & (u_1,\cdots,u_{d-1},v_1,\cdots,v_{d-1},w_1) \in (\bC^k)^{2d-2} \times \bC^N \ : \ \\
			                                & \sum_{j=1}^{m-1} u_j v_{m-j} = 0, \ m=2,\cdots,d-1,                               \\ &
			\sum_{j=1}^{d-1} u_j v_{d-j} + G(w_1) \neq 0
			\Big\}.
		\end{aligned}
	\end{equation}
	Also $\chi'_{g_k}$ is isomorphic to  the  variety defined by the same equations except that the last equation is zero (rather than non zero) and $\chi''_{g_k}$ is also isomorphic to the same variety, but with the last equation removed completely.
\end{lemma}
\begin{proof}
	The jet
	\begin{equation} \label{eqnuvwoft}
		(u(t),v(t),w(t)) \in \scrX(\bC^k) \times \scrX(\bC^k) \times \scrX(\bC^N) = \scrX(\bC^k \times \bC^k \times \bC^N),
	\end{equation}
	where
	\[
		u(t)=u_0 + \cdots + u_dt^d, \  v(t)=v_0+\cdots+v_dt^d, \  w(t)=w_0+\cdots+w_dt^d,
	\]
	with
	\[u_0,\cdots,u_d,v_0,\cdots,v_d \in \bC^k, \ w_0,\cdots,w_d \in \bC^N
	\] lies in $\chi_{g_k}$ if and only if $u_0=0$, $v_0=0$ and the coefficient $t^m$ of $u(t)v(t)$ vanishes for each $m <d$ and the coefficient $t^d$ of $u(t)v(t)+G(w(t))$ is non-zero. These equations hold if and only if $u_0=0$, $v_0=0$,
	$\sum_{j=1}^{m-1} u_j v_{m-j}=0$ for each $m=2,\cdots,d-1$
	and $\sum_{j=1}^{d-1} u_j v_{d-j}+G(w_1) \neq 0$ which corresponds to the variety in the statement of our Lemma.
	A similar argument holds for $\chi'_{g_k}$ and $\chi''_{g_k}$.
\end{proof}

\begin{defn} \label{defnchiprimeg}
	For each $i=0,\cdots,d$ we define $\chi'_{g_k,i} \subset \chi'_{g_k}$ to be the subspace of elements
	$(u(t),v(t),w(t))$ as in \eqref{eqnuvwoft} so that $u_0,\cdots,u_{i-1} = 0$ where
	\begin{equation} \label{eqnuoft}
		u(t) = u_0 + u_1 t + \cdots + u_d t^d, \ u_j \in \C^k, \ j=0,\cdots,d.
	\end{equation}
	Similarly for each $i=0,\cdots,d-1$ we define $\chi''_{g_k,i} \subset \chi''_{g_k}$ to be the subspace of elements
	$(u(t),v(t),w(t))$ as in \eqref{eqnuvwoft} so that $u_0,\cdots,u_{i-1} = 0$ from \eqref{eqnuoft}.
\end{defn}

\begin{lemma} \label{lemmaStrataFibration}
	For each $i= 0,\ldots,d$,
	$\chi'_{g_k,i}-\chi'_{g_k,i+1}$ is isomorphic to the total space of an affine bundle\footnote{By an \emph{affine bundle} we mean a morphism which, Zariski locally over the base, is isomorphic to a projection map $\bC^m \times U \to U$.} over $\bC^k-0$.
	Similarly for each $i = 0,\ldots,d-1$,
	$\chi''_{g_k,i}-\chi''_{g_k,i+1}$ is isomorphic to the total space of an affine bundle over $\bC^k-0$.
\end{lemma}
\begin{proof}
	Fix a value of $i$. We use the vectors $u_1,\cdots,u_k,v_1,\cdots,v_d,w_0,\cdots,w_d$ from the proof of Lemma \ref{LemmaChiComputation}.
	Let $u_{j,l}$ be the $l$th coordinate of $u_j$ for each $j \geq i$ and each $l \in \{1,\cdots,k\}$.
	Similarly define $v_{j,l}$ for each such $j,l$.
	Define
	\begin{equation}
		v_{j,\neq l} := (v_{j,1},\cdots,v_{j,l-1},v_{j,l+1},\cdots,v_{j,k})
	\end{equation}
	for each $j \geq i$ and $l \in \{1,\cdots,k\}$.

	We claim that the map
	\begin{equation}
		u_i: \chi'_{g_k,i}-\chi'_{g_k,i+1} \to \C^k-0
	\end{equation}
	is the projection map witnessing the domain as the total space of an affine bundle.
	To do this, we will construct a trivialisation of $u_i$ in the region where $u_{i,l} \neq 0$ for some fixed $l \in \{1,\cdots,k\}$.
	By Lemma \ref{LemmaChiComputation}, we have $(\chi'_{g_k,i}-\chi'_{g_k,i+1}) \cap \{u_{i,l} \neq 0\}$ is isomorphic to
	\begin{equation} \label{eqnMainChiEqn2}
		\begin{aligned}
			\bC^{(d-1)N+2k} \times & \Big\{(u_{i+1},\cdots,u_{d-1},v_1,\cdots,v_{d-1},w_1)  \in (\bC^{k})^{2(d-1)-i} \times \bC^N \ : \
			u_{i,l} \neq 0,                                                                                                                                      \\ &
			v_{m-i,l} = -\frac{1}{u_{i,l}}\Big(\sum_{\underset{l' \neq l}{l'=1}}^k u_{i,l'}v_{m-i,l'} + \sum_{j=i+1}^{m-1} u_j v_{m-j}\Big), \ m=i+1,\cdots,d-1, \\ &
			v_{d-i,l} = -\frac{1}{u_{i,l}}\Big(\sum_{\underset{l' \neq l}{l'=1}}^k u_{i,l'}v_{d-i,l'} + \sum_{j=i+1}^{d-1} u_j v_{d-j}+G(w_1)\Big) \Big\}.
		\end{aligned}
	\end{equation}
	Therefore our Zariski local trivialisation is given by:
	\begin{equation}
		\begin{aligned}
			 & (\chi'_{g_k,i}-\chi'_{g_k,i+1}) \cap \{u_{i,l} \neq 0\} \to \bC^{(k-1)(d-i) + kd + Nd} \times (\bC^{l-1} \times \bC^* \times \bC^{k-l-1}),
			\\  & (u_i,\cdots,u_d,v_1,\cdots,v_d,w_1,\cdots,w_d) \to \\ &  ((u_{i+1},\cdots,u_d,v_{1,\neq l},\cdots,v_{d-i,\neq l},v_{d-i+1},\cdots,v_d,w_1,\cdots,w_d),u_i)
		\end{aligned}
	\end{equation}
	A similar argument holds for $\chi''_{g_k,i}-\chi''_{g_k,i+1}$.
\end{proof}

\begin{lemma} \label{lemmaAlluVanishComputation}
	We have $\chi'_{g_k,d} = \bC^{(d-1)N + kd + k} \times G^{-1}(0)$
	and $\chi''_{g_k,d-1}=\C^{dN+kd+2k}$.
\end{lemma}
\begin{proof}
	All the $u_i$ variables are $0$ for $i < d$ and so we just get a description of $\bC^{(d-1)N+dk + k} \times G^{-1}(0)$ by Lemma \ref{LemmaChiComputation}.
	A similar argument holds for $\chi''_{g_k,d-1}$.
\end{proof}

\begin{lemma} \label{lemmaChiFromChiPrime}
	$\chi_{g_k}= \chi''_{g_k}-\chi'_{g_k}.$
\end{lemma}
\begin{proof}
	This follows directly from Lemma \ref{LemmaChiComputation}.
\end{proof}

\begin{lemma} \label{lemmaHodgeLES}
	Let $V$ be a quasi-projective variety, $U \subset V$ a Zariski open subset and $Z=V-U$. If the $e$-weight-graded piece $\Gr^W_e(H_c^*(Z);\bQ)$ of compactly supported cohomology vanishes, then for any $j$,
	\begin{equation}
		\Gr^W_e(H_c^j(U);\bQ) \cong \Gr^W_e(H_c^j(V);\bQ).
	\end{equation}
	Also, if $\Gr^W_e(H_c^*(U);\bQ) \cong 0$, then for any $j$,
	\begin{equation}
		\Gr^W_e(H_c^j(V);\bQ) \cong \Gr^W_e(H_c^j(Z);\bQ).
	\end{equation}
	If $\Gr^W_e(H_c^*(V);\bQ) \cong 0$, then for any $j$,
	\begin{equation}
		\Gr^W_e(H_c^j(Z);\bQ) \cong \Gr^W_e(H_c^{j+1}(U);\bQ).
	\end{equation}
\end{lemma}
\begin{proof} This uses the fact that $e$-weight graded pieces of compactly supported cohomology fit into the usual long exact sequence:
	\begin{equation}
		\begin{aligned}
			\cdots \to
			\Gr^W_e(H_c^
			{j-1}(V))
			\to \Gr^W_e(H_c^{j-1}(Z)) \to \Gr^W_e(H_c^j(U)) \\
			\to \Gr^W_e(H_c^j(V))
			\to \Gr^W_e(H_c^j(Z))
			\to \Gr^W_e(H_c^{j+1}(U)) \to \cdots
		\end{aligned}
	\end{equation}
\end{proof}

\begin{lemma}\label{lemmaHodgeOfBirational}
	If $V$ is birational to $V'$ and both are of dimension $m$, then $\Gr^W_{2m-1}(H^j_c(V);\bQ) = Gr^W_{2m-1}(H^j_c(V');\bQ).$
\end{lemma}
\begin{proof}
	Use Lemma \ref{lemmaHodgeLES} twice with the common open subset of $V$ and $V'$.
\end{proof}

\begin{lemma} \label{lemmaOddWeightVanishing}
	$\Gr^W_e(H_c^*(\chi''_{g_k};\bQ)=0$ for each odd $e$.
\end{lemma}
\begin{proof}
	This follows from  Lemmas \ref{lemmaStrataFibration}, \ref{lemmaAlluVanishComputation} and \ref{lemmaHodgeLES} combined with the fact that \begin{equation}
		\Gr^W_e(H_c^*(\bC^m;\bQ)=0, \quad \forall \ m \in \bN,
	\end{equation}
	$\chi'_{g,0} = \chi'_g$ and $\chi''_{g,0} = \chi''_g$.
\end{proof}

We now specialise the discussion of this section to the situation described after Lemma \ref{lem:add to make isolated}, i.e. we take the polynomial $g$ to be the polynomial $f$ constructed there, whose leading order term $G$ of degree $d$ is the polynomial $F$ whose zero-set was birational to $\bC$ multiplied by our abelian variety $A$.  We thus conclude:

\begin{lemma} \label{lemmafkHodgeComputation}
	Let $m = (d-1)N+dk + k+1$.
	Then $Gr^W_{\dim(A)-1+2m}(H_c^{\dim(A)+2m}(\chi_{f_k};\bQ))$ is isomorphic to the Tate twisted pure Hodge structure $H^{\dim(A)-1}(A;\bQ)(-m)$.
\end{lemma}
\begin{proof}
	We use Lemmas \ref{lemmaStrataFibration}, \ref{lemmaAlluVanishComputation},
	\ref{lemmaChiFromChiPrime},
	\ref{lemmaHodgeLES} and \ref{lemmaHodgeOfBirational}
	and \ref{lemmaOddWeightVanishing}
	combined with the fact that $\bC^{(d-1)N+dk + k} \times G^{-1}(0)$ is birational to $\bC^m \times A$.
\end{proof}

\section{Proof of the Main Theorem}

The main theorem \ref{maintheorem} follows immediately from the following theorem combined with Lemmas \ref{equivalentstandard} and \ref{lemmastabilisationofNND}. We recall the polynomial $f$ constructed after Lemma \ref{lem:add to make isolated}, and its stabilisations $f_k$.

\begin{theorem} \label{theoremfknotstablydegenerate}
	For each $k \in \bN$, any polynomial equivalent to $f_k$ is \emph{not} Newton non-degenerate.
\end{theorem}
\begin{proof}
	Suppose, for a contradiction, $\check{f}_k$ is Newton non-degenerate where $\check{f}_k$ is equivalent to $f_k$ for some $k \in \bN$.
	Let $D \in \bN$ be greater than the degree of $\check{f}_k$ and hence also greater than the sum of the coordinates of every integer lattice point of every compact face of $\Gamma_+(\check{f}_k)$.
	Let $(f_\tau)_{\tau \in T}$ be the space of Newton non-degenerate polynomials with isolated singularities at $0$, with the same Newton polyhedron as $\check{f}_k$ and whose degree is bounded by $D$. Let $\tau_0 \in T$ be such that $f_{\tau_0} = \check{f}_k$. Then by Lemma \ref{lemmaTopTrivual},
	$(R\pi_{f_T})_!\bQ_{\chi_{f_T}}$ is a locally constant sheaf.
	Therefore we can make appropriate choices so that
	\begin{equation}
		H^*_c(\chi_{f_\tau};\bQ), \tau \in T
	\end{equation}
	together with their weight and Hodge filtrations
	is a polarised variation of mixed Hodge structures (for any local trivialisation of $(R\pi_{f_T})_!\bQ_{\chi_{f_T}}$).

	By Lemmas \ref{lemmaChangeOfCoordinatrd} and \ref{lemmafkHodgeComputation}, there exists $i,m \in \bN$ so that $\Gr^W_i(H^{i+1}_c(\check{f}_k);\bQ)$ is isomorphic to $H^{2g-1}(A;\bQ)(-m)$. Since the N\'eron-Severi group of $A$ has rank $1$ by property \eqref{itemNScondition} of Lemma \ref{lemmaconditionsforA}, we can assume this is an isomorphism of polarised Hodge structures after rescaling the polarization.
	Hence $\Gr^W_i(H^{i+1}_c(\check{f}_k);\bZ)(m+g-1)$ is isomorphic to $H^{2g-1}({A'}^\vee;\bZ)(g-1) = H^1(A';\bZ)$ along with its polarization for some variety $A'$ whose dual ${A'}^\vee$ is isogenous to $A$. Let $\delta \in \bN^g$ be such that $A' \in \scrA_{g,\delta}$.

	Appealing to Remark \ref{rmk:abvar-as-Hodge}, that $\scrA_{g,\delta}$ is also the moduli space of level 1 polarised integral Hodge structures of type $\delta$, we therefore we have a period map
	$T \to \scrA_{g,\delta}$ sending each $\tau$ in $T$ to
	$\Gr^W_i(H^{i+1}_c(\chi_{f_\tau};\bZ)(m+g-1)$ with the property that $\tau_0$ gets sent to $A'$.

	The graph of this map is an algebraic subvariety of $T \times \scrA_{g,\delta}$ by \cite{cattani1995locus}. Hence the map $T \to \scrA_{g,\delta}$ is an algebraic morphism.

	Therefore by property \eqref{itemNotInRational} of Lemma \ref{lemmaconditionsforA} this map must be constant. But this contradicts
	part \eqref{itemNoConstantSummandOfToric} of Lemma \ref{lemmaconditionsforA} due to the fact that
	$\pi_{f_T}$
	is defined over $\bQ$.
\end{proof}

\begin{Remark}
	Note that the only properties of Newton non degenerate polynomials that we used in the proof of Theorem \ref{maintheorem} are the following:
	\begin{enumerate}
		\item The space of Newton non degenerate polynomials is a countable union of rational varieties $T_i$, $i \in \bN$.
		\item For each such $T_i$,
		      $(\pi_{f_{T_i}})_!\bQ_{\chi_{f_{T_i}}}$ is a locally constant sheaf.
		\item If $f$ is a Newton non degenerate polynomial then $f(z) + w^2$ is also Newton non degenerate (Lemma \ref{lemmastabilisationofNND}).
	\end{enumerate}
	The only part of the proof that needs adjusting is that $\scrS(\bQ)$ in part \eqref{itemNoConstantSummandOfToric} of Lemma \ref{lemmaconditionsforA} is replaced with $\scrS(\bK)$ where $\bK \subset \bC$ is a countable subfield for which the varieties $T_i$, $i \in \bN$ are defined.
\end{Remark}

\section{Towards an algorithm}\label{sec:algorithm}

It is natural to ask for an \emph{explicit} stably degenerate hypersurface singularity. Whilst it does not seem practical to extract one from the above argument, there is in principle an algorithm to do so, following Mumford's description of the homogeneous co-ordinate ring of an abelian variety \cite{Mumford1} and of explicit equations cutting out the moduli space of abelian varieties \cite{Mumford2}. The crucial fact is that the Riemann relations which define the homogenous ideal of the canonical projective embedding of $A$ are polynomial functions of $\theta$-values derived from the defining period matrix which are defined by universal algebraic identities over $\bQ$, and that the equations in theta-nulls cutting out $\scrA_{g,\delta}$ are defined over $\bQ$.

Let $\Aut(\bC)$ denote the group of field automorphisms of $\bC$. For $\sigma \in \Aut(\bC)$ and $X$ a complex projective variety, we write $X_{\sigma}$ for the corresponding base-change, i.e. for the fibre product
\[
	\begin{tikzcd}
		X_{\sigma} \ar[r] \ar[d] & X \ar[d] \\
		\mathrm{Spec}(\bC) \ar[r,"\sigma"] & \mathrm{Spec}(\bC)
	\end{tikzcd}
\]
Separately, $\sigma$ defines an automorphism $\hat\sigma: \scrA_g(\bC) \to \scrA_g(\bC)$ of the (complex points of the) coarse moduli space $\scrA_g$ of principally polarised abelian varieties, since $\scrA_g$ is defined over $\bQ$.

\begin{lemma}
	If $A$ is a principally polarised complex abelian variety with moduli point $[A] \in \scrA_g(\bC)$, then $[A_{\sigma}] = \hat\sigma[A]$.
\end{lemma}

\begin{proof}
	If $\bA_g(\cdot)$ is the functor on the category of schemes for which $\scrA_g$ is the coarse moduli space, then the map $\bA_g(\Spec(\bC)) \to \bA_g(\Spec(\bC))$ induced by $\sigma$ is obtained from fibre product over the map $\sigma: \Spec(\bC) \to \Spec(\bC)$ by  construction of the moduli functor. Now apply the functor of points $\bA_g \to \scrA_g$ to pass from $\sigma$ to $\hat\sigma$.
\end{proof}

\begin{remark}
	Note that Siegel upper half-space is not preserved by the Galois action; if $A$ is represented by a matrix $M \in \mathfrak{h}_g$ , then $A_{\sigma}$ will not be represented by the co-ordinate-wise transform $\sigma(M)$ (which need not lie in $\mathfrak{h}_g$), and there seems to be no elementary description of a period matrix which does define $A_{\sigma}$. Rather, algebraic invariants such as modular functions of $[A]$ transform well under $\sigma$.
\end{remark}

\begin{remark}
	Concretely, a principally polarised abelian variety $L \to A$ is determined by a point $\tau_A$ in Siegel upper half-space $\mathfrak{h}_g$. One has the theta-functions
	\[
		\theta \left[\begin{matrix} a \\ b \end{matrix} \right](z,\tau) = \sum_{m \in \bZ^g} \exp \left(\pi i (m+a)^t \tau (m+a) + 2i\pi (m+a)^t (z+b) \right)
	\]
	and associated values $\Theta_a(z) = \theta\left[\begin{matrix} a \\ 0 \end{matrix} \right](nz,n\tau)$ which give a basis of $H^0(A,L^n)$; here $a \in (\bZ/n\bZ)^g$, $L$ is the principal polarisation. For $n\geq 3$, $L^n$ is ample, and for $n\geq 4$ the corresponding projective image is cut out by quadrics \cite{Mumford1}.  The Riemann bilinear relations are universal equations over $\bQ$ which govern the constants $C_{ab}^c$ underlying  multiplication $\Theta_a(z)\Theta_b(z) = C_{ab}^c \Theta_c(2z)$ and thence the quadratic equations cutting out the kernel of $\Sym^2(H^0(L^n)) \to H^0(L^{2n})$.
\end{remark}

Given $[A] \in \scrA_g$, and the embedding $A \in \bP(H^0(L^4))^{\vee} \cong \bP^{4^g-1}$, we have the associated affine cone $C(A) \subset \bC^{4^g}$; we use the same notation to denote a hypersurface with an isolated singularity given by adding a polynomial $h$ of high degree as in Lemma \ref{lem:add to make isolated}. Under an appropriate choice of linear projection $\pi: \bC^{4^g} \to \bC^{g+1}$, $\pi(C(A))$ is a hypersurface which is cut out by an equation $F_{A,\pi,h}$ generating a principal ideal $I(A,\pi,h) \subset \bC[x_0,\ldots,x_g]$.

\begin{remark}
	Given the equations for $A$ in projective space, and hence for the affine cone $C(A)$ depending on the choice of $g$, and given an explicit projection $\pi$, one can use elimination theory / Gr\"obner basis techniques to obtain the polynomial $F_{A,\pi,g}$.
\end{remark}

Note that a field automorphism $\sigma$ also gives a new, conjugate linear projection $\pi^{\sigma} = \sigma \circ \pi \circ \sigma^{-1}$ (which agrees with the original one if $\pi$ is defined over $\bQ$), and can also be applied to the coefficients of the polynomial $h$ to give a new $h^{\sigma}$.

\begin{lemma} \label{lemmastableautomorphism}
	We have
	\begin{equation}
		\sigma(I(A,\pi,h)) = I(A^{\sigma},\pi^{\sigma},h^{\sigma}).
	\end{equation}
	Furthermore, the hypersurface defined by $F(A,\pi,h)$ has a Newton non-degenerate stabilisation if and only if that defined by $F(A^{\sigma},\pi^{\sigma},h^{\sigma})$ does.
\end{lemma}

\begin{proof}
	The first statement is follows from the discussion anove. For the second, Newton non-degeneracy is defined by non-vanishing of various algebraic equations in the derivatives, and so is preserved by field automorphisms.
\end{proof}

It follows that to have an in-principle algorithm to construct an example of a stably degenerate hypersurface, it suffices to have a way of picking an abelian variety which satisfies our basic constraints, in particular which does not lie on any rational curve in $\scrA_g$. So:

\begin{enumerate}
	\item Mumford \cite{Mumford2} gives equations for a projective embedding of $\scrA_g$, in co-ordinates which are modular functions of periods, so one can take an affine chart $U$ with $\scrA_g  \supset U \subset \bC^N$ (e.g. $N = 4^g-1$) cut out by explicit equations over $\bQ$;
	\item pick a projection defined over $\bQ$, say $\phi:  \bC^N \supset U \to \bC^D$, where $D = \dim_{\bC}(\scrA_g) = g(g+1)/2$ (so $\phi$ is the action on $\bC$-points of a projection $\scrA_g(\bQ) \supset U(\bQ) \to \bQ^D$) so that the map $\phi|_U$ is finite and dominant;
	\item  take distinct primes $\{p_1,\ldots,p_D\}$ and a point $P = (e^{p_1},\ldots, e^{p_D}) \in \bC^D$ with co-ordinates which are algebraically independent over $\bQ$;
	\item the point $P$ must lie in the image of $\phi|_U$ since there exists a point $Q$ in its image with algebraically independent coordinates and so if $\check{\sigma} \in \Aut(\bC)$ sends $Q$ to $P$, then $\check{\sigma}$ sends a point in the preimage of $Q$ to a point in the preimage of $P$; similar reasoning ensures that $\phi|_U^{-1}(P)$ is finite;
	\item therefore an explicit point $[A]$ in $(\phi|_U)^{-1}(P)$ can now be found by Gr\"obner basis / elimination theory techniques; since we have equations for $\scrA_g$ so this is a matter of finding a solution to a zero-dimensional polynomial system due to the fact that $\phi|_U$ is finite; this point has coordinates living in an explicit finite field extension of $\bQ$;
	\item the polynomial $F_{A,\pi,h}$ can then be built explicitly from these coordinates since Mumford \cite{Mumford2} has explicit quadratic equations cutting out $A$ in projective space using theta function values.
\end{enumerate}

Now $F_{A,\pi,h}$ in the proposed algorithm above is stably degenerate for the following reason:
By a slight generalization of Lemma \ref{lemmaconditionsforA}, we can find a point $R$ in the image of $\phi|_U$ with algebraically independent coordinates so that every abelian variety in its preimage satisfies the properties of Lemma \ref{lemmaconditionsforA}.
We choose $\sigma \in \Aut(\bC)$ so that it sends $P$ to $R$.
Hence $[A^\sigma]$ lies in the preimage of $R$.
Therefore by Lemma \ref{lemmastableautomorphism} combined with Theorem \ref{theoremfknotstablydegenerate} and Lemmas \ref{equivalentstandard} and \ref{lemmastabilisationofNND}, we get $F_{A,\pi,h}$ is stably degenerate.

The polynomial $F_{A,\pi,h}$ can be explicitly defined over an explicit finite field extension of $\bQ$. However, if we wish to write down its coefficients as elements of $\bC$, then we have to use numerical approximation techniques to embed this finite field extension into $\bC$.

Gr\"obner basis algorithms for ideals in $N$ variables run, in a worst case scenario, in doubly exponential time (and may turn quadratic relations into ones defined by polynomials of exponential in $N$ degree), so implementing this is far from practical. Moreover, since $\scrA_g$ is cut out by at least $4^g-g(g+1)/2-1$ quadratics, and $g \geq 16$, the degree of the hypersurface birational to the abelian variety $A$ is huge, and any eventual polynomial arising from this method would likely have millions of monomial terms.

\begin{remark}
	It is natural to ask if the singularity constructed after Lemma \ref{lem:add to make isolated},
	or via the algorithm above, after stabilisation and coordinate change, can be deformed through a $\mu = \text{constant}$ family to a nearby Newton non-degenerate one.
	It seems likely that one can prove this is not possible by combining the methods of this paper with
	Floer theory. In particular one would use the proposed isomorphism $H^*_c(\chi_{f_\tau}) \cong HF^*(\phi_{f_\tau}^d)$, with $\phi_{f_\tau}$ the monodromy of $f_\tau$ and $d$ the contact order of the jets, as in  \cite[Theorem 7.2.69]{McLean:arcsurvey},
	together with ideas from \cite{fernandez2024symplectic}
	to show that $(R\pi_{f_T})_!\bQ_{\chi_{f_T}}$ is locally constant for a $\mu = \text{constant}$ deformation $(f_\tau)_{\tau \in T}$. This would give a strengthening of the main result.
\end{remark}

\bibliographystyle{alpha}
\bibliography{references}

\end{document}